\documentclass[preprint,12pt]{elsarticle}

\usepackage{amsthm,amsmath,amsfonts,amssymb}
\usepackage[colorlinks,citecolor=blue,urlcolor=blue]{hyperref}
\usepackage{graphicx}

\numberwithin{equation}{section}

\theoremstyle{plain}
\newtheorem{theorem}{Theorem}[section]
\newtheorem{lemma}[theorem]{Lemma}

\theoremstyle{remark}

\newtheorem{proposition}[theorem]{Proposition}
\newtheorem{corollary}[theorem]{Corollary}

\newtheorem*{remark}{Remark}

\renewcommand\P{{\mathbb P}}
\newcommand\R{{\mathbb R}}
\newcommand{\crit}{{\rm Crt}}
\newcommand{\I}{{\rm Id}}
\newcommand{\J}{{\rm J}}
\renewcommand{\u}{{\rm u}}
\def\cT{\mathcal{T}}
\newcommand\E{{\mathbb E}}
\newcommand{\Var}{{\rm Var}}
\newcommand{\Cov}{{\rm Cov}}
\newcommand\X{X(\cdot)}
\newcommand\cC{\mathcal{C}}
\newcommand{\balpha}{\boldsymbol \alpha}
\newcommand\ve{\varepsilon}
\newcommand\cA{\mathcal{A}}
\newcommand{\indicator}{\mathbf{1}}
\newcommand{\rX}{\boldsymbol X}
\newcommand{\rY}{\boldsymbol Y}
\newcommand\N{{\mathbb N}}
\newcommand\cH{\mathcal{H}}
\newcommand\C{{\mathbb C}}
\newcommand{\hB}{{\widehat B}}
\newcommand{\<}{\big\langle}
\renewcommand{\>}{\big\rangle}
\newcommand\diag{\mbox{diag}}
\newcommand\ct{\mathtt{(cs)}}
\newcommand{\bY}{{ \mathbf Y}}
\newcommand\Z{{\mathbb Z}}
\newcommand\dd{{\mathrm d}}

\begin{document}

\begin{frontmatter}

\title{Multivariable CLT for critical points}

\author[1]{Jean-Marc Aza\"{\i}s}
\ead{jean-marc.azais@math.univ-toulouse.fr}

\author[2]{Federico Dalmao}
\ead{fdalmao@unorte.edu.uy}

\author[3]{C\'eline Delmas}
\ead{celine.delmas.toulouse@inrae.fr}

\address[1]{IMT, Universit\'e de Toulouse, France.}
\address[2]{DMEL, CENUR Litoral Norte, Universidad de la Rep\'{u}blica, Uruguay}
\address[3]{MIAT, INRAE, Universit\'e de Toulouse, France.}

\begin{abstract}
We prove a multivariate central limit theorem for the numbers of critical points with all possible indexes which lie above a level of a non-necessarily isotropic Gaussian random field.
We prove the non-degeneracy of the limit joint distribution in the isotropic case.
We also consider the degenerate case when the value is not restricted to lie above a level.
We extend, to the non-isotropic framework, known results by Estrade \& Le\'on and Nicolaescu for the Euler characteristic of an excursion set and for the total number of critical points of Gaussian random fields.
Though we use the classical tools of chaotic expansions and fourth moment theorem, our proof of the non-degeneracy of the limit distribution does not focus on the explicit description of the lowest order chaotic components but it transforms some convenient Hermite coefficients of arbitrary order into functions of the eigenvalues of a Gaussian Orthogonal Ensemble (GOE) random matrix and use the classical Laplace Method to conclude.
\end{abstract}


\end{frontmatter}
\newpage
\tableofcontents

\section{Introduction}
The behavior of critical points of random fields has been widely studied from different points of view.
The first seminal works date back to the papers by Nosko \cite{nosko1, nosko2}, Lindgren \cite{lindgren}, Belyaev and Piterbarg \cite{pit1, pit2}, Hasofer \cite{hasofer}, the books by Adler \cite{adler} and by Piterbarg \cite{pit3}.
They mainly focus on local maxima of Gaussian random fields, their positions and their mean number.
In particular, a one-term approximation for the mean number of local maxima above a high level is given by Hasofer \cite{hasofer} and Adler \cite{adler} for a stationary Gaussian field.
A more accurate asymptotic expansion is then obtained by Delmas \cite{delmas1998} and Aza\"{\i}s and Delmas \cite{AD2002}, who also give a good approximation for the distribution of the maximum of a zero-mean stationary Gaussian field. It is worth mentioning that these studies are related to those concerning the Euler characteristic \cite{adler, adlertaylor, taylortakemuraadler, marioyluigi}. The main applications of such results concern the detection of peaks in a random field. For example, we can mention papers on the detection of activation zones in the human brain \cite{worsley, worsley2, taylorworsley}, on applications in astronomy \cite{torres, vogeley}, on the detection of genes on a chromosome \cite{azaiscierco, azaisdelmasrabier, rabierdelmas}. Local maxima of random fields are also used to define a crest in the modeling of random sea waves, see \cite{azais2005}, \cite[Ch.11]{marioyluigi} and references therein.

More recent works study the spread of the critical points of a Gaussian random field.
In particular, it is interesting to know if there is attraction or repulsion among them.
The answer can depend on the indexes, i.e., on the number of negative eigenvalues of the Hessian, of the considered critical points.
The seminal work by Beliaev, Cammarota and Wigman \cite{BCW} considers the particular case of the \emph{Random Plane Wave}, a popular Gaussian field defined on $\R^2$.
These results have been extended to more general planar random fields in \cite{BCW2} and \cite{LLR} and to any dimension in \cite{AD2022}.
Another direction is to study the number of critical points above some level and with given indexes.
This is the object of the papers by Auffinger, Ben Arous and Čern\'{y} \cite{ABAC} and by Auffinger and Ben Arous \cite{ABA}, where an exponential behavior of the number of critical points is computed as a function of the index and of the level.
These results have direct applications to the minimization of likelihood functions in large dimension since there is a critical region (for the level) in which only minima can be found (at the logarithmic scale).
The last direction is to establish central limit theorems (CLT) on a family of growing sets.
The first paper on this is by Estrade and Le\'on \cite{EL} and considers the Euler characteristic of the excursion set above the level $u$ of a real-valued isotropic Gaussian random field over a set $\cT\subset\R^d$.
This quantity is approximated by an alternate sum of the numbers $\crit^k (u,\cT)$ of critical points with index $k$ above $u$ within $\cT$.
More precisely, the modified Euler characteristic is defined by 
\begin{equation*} 
	\Phi_u (\cT) := \sum _{k=0}^d  (-1)^{d-k}\crit^k (u,\cT).
\end{equation*}
The paper uses the Hermite representation of the number of critical points and the fourth moment theorem \cite{Nou:Pec:Pod,Nou:Pec}.
In another paper, Nicolaescu \cite{nico} studies the total number of critical points
\[
   \sum _{k=0}^d  \crit^k (-\infty,\cT)
\]
and establishes a central limit theorem using analogous methods.
This result has been extended to the study of $ \crit^k (-\infty,\cT)$ by \cite{chevalier}.
\medskip

In this paper we aim at obtaining a deeper understanding of the (joint) distribution of the different kinds of critical points of a centered stationary Gaussian field defined on $\R^d$.
More precisely, we consider the numbers of critical points within a set $\cT$ with all possible indexes $k=0,1,\cdots,d$ whose values lie above a level $u\in\R\cup\{-\infty\}$, that is to say that we consider the distribution of the random vector
\[
 \Big( \crit^0 (u,\cT),\cdots, \crit^d (u,\cT) \Big).
\]
Our contribution is threefold: 
(i.) We establish the multivariate CLT as $\cT$ grows to $\R^d$ for the vector above with a finite limit variance matrix.
Here we resort on chaotic expansions and on the celebrated Fourth Moment Theorem.\\
(ii.) We remove the usual isotropy assumptions on the field.
Actually, a careful reading of the literature (see e.g.: \cite{EL} or \cite{nico}) shows that isotropy is used only to state the finiteness and continuity w.r.t.
the level of the second moment of the number of critical points or of some related quantity.
At this point we are inspired by \cite{AL,EF}, see Proposition \ref{p:prop} below.
\\
(iii.) We discuss the non-degeneracy of the limit joint distribution.
This is the most novel part of the paper since we lack systematic methods, see the related discussion in \cite{Gass-esp}.
Even if we resort on chaotic expansions to get the non-degeneracy of the limit distribution for $u\neq-\infty$ in the isotropic case, since the dimension is arbitrary, we are led to deal virtually with all the chaotic components and not only with the first ones as is the usual case in the literature.
As expected, at some point we have to delve into Hermite coefficients.
We are able to discard the major part of these coefficients and to focus on a particular class which we can manipulate so that we avoid their explicit computation by transforming the problem into the study of the linear independence of some functions of the eigenvalues of the Gaussian  Orthogonal Ensemble (GOE) matrices.
We use the classical Laplace method to conclude.
\smallskip
     
The organisation of the paper is as follows.
We present the general framework and the main theorem in Section \ref{s:gmt}.
In Section \ref{s:H}, we introduce chaotic expansions, one of our main tools.
Section \ref{s:proof} is dedicated to the proof of the main result.
Finally,  \ref{appA}, \ref{appB}, \ref{s:app:geman}, \ref{a:pqcero} and \ref{a:somel} present auxiliary material and some postponed proofs.
{\bf Notation:} 
N. D. means \emph{non-degenerate}: for a matrix it means no zero eigenvalue and for a multivariate Gaussian distribution, it means no zero eigenvalue for its variance-covariance matrix;
$\ct$ stands for an unimportant positive constant whose value may change from line to line;
$\phi_k$ is the standard normal density in $\R^k$; $\mathbb{S}^{k-1}$ is the unit sphere in $\R^k$; $[q]=\{1,\cdots,q\}$;
$\simeq$ is the equivalence of functions; $\cal F$ is the Fourier transform;
$\Var$ is the variance-covariance matrix of a random vector; $\Cov$ is the covariance matrix between two random vectors;
we choose some norms that are denoted $\|\cdot\|$ for the spaces of matrices and order $3$ and $4$ tensors;
$e^{\mathcal D}_k:k=0,1,\cdots,\mathcal D$ denote the canonical vectors in $\R^{\mathcal D}$.
A  list of  main notation is presented in  \ref{app:notation}.
\section{Generalities and main result} \label{s:gmt}
Let $\X$ be a real valued random field defined on $\R^d$,  $d\geq 1$.
As soon as they exist, write $X'(\cdot)$ for the gradient of $X(\cdot)$ and $X''(\cdot)$ for its Hessian matrix.
Consider the following assumptions:
\begin{enumerate}
\item[] (A1) : The random field $\X$ is centered, stationary and Gaussian with $C^2$ paths.
Assume that $\Var (X'(t))$ is N.D.
\end{enumerate}
Denote $r(\cdot)$ for the covariance function of $\X$, that is, $r(t)=\E\big(X(0)X(t)\big)$, $t\in \R^d$.
\begin{enumerate}
\item[] (A2) : {\bf Geman's condition.}
For sufficiently small Borel set $B \subset \R^d $, $\E\big((N(v,B))^2\big)$ is continuous as a function of $v$ at zero, where 
\begin{equation} \label{a:1}
N(v,B):= \# \{ t \in B : X'(t) = v\}, \hspace{0.5cm} v,t\in \R^d.
\end{equation}

\item[] (A3) : {\bf Arcones' condition.}
Defining  
\begin{equation} \label{e:psi}
  \Psi(t) := \max\left\{\Big|\frac{\partial^j r(t)}{\partial t_m}\Big|
: m\in [d]^j, 0\leq j\leq 4\right\},
\end{equation}
it holds that $\Psi(t)\to0$, as $\|t\|\to\infty$, and $\Psi\in L^1(\R^d)$.
\item[] (A4) :  {\bf Isotropy and spectral condition.} The random field $\X$ is isotropic and admits a spectral density which is positive in a neighborhood of $0$.
\end{enumerate}
\medskip

\begin{remark}
Assumption {\rm (A1)} establishes minimal regularity and non-degeneracy conditions for the underlying random field.
Compared to previous works \citep{EL,nico,AL}, it is worth pointing out that our work does not ask for isotropy, except for setting the non-degeneracy of the limit distribution.
Besides, note that {\rm (A1)} implies by \citep[Prop.2.1.]{aal} that the sample paths of $\X$ are almost surely Morse: 
\[
 \P\big\{ \exists t \in \R^d : X'(t)=0,\  \det (X''(t) ) =0\big\} =0.
\] 
Assumption {\rm (A1)} is met, for instance, if the covariance is $C^4$ with some logarithmic regularity of its fourth derivative at zero.
Assumption {\rm (A2)} is natural according to previous works \cite{geman,K:L-geman,AL}.
Proposition \ref{p:prop} below gives sufficient conditions for {\rm (A2)} to hold.
Assumption {\rm (A3)} provides a simple way to unify the necessary integrability of the covariance function and its derivatives regarding the proof of the asymptotic normality.
Assumption {\rm (A4)} simplifies the computations involved in the lower bounds for the limit variance.
Note that Assumption {\rm (A3)} implies that the spectral density exists so that the second part of Assumption {\rm (A4)} can be reduced to the positiveness of the spectral density around $0$.
Finally, it is worth mentioning that these four assumptions are met for classical families of Gaussian fields as: Bargmann-fock, Mat\'ern (with $\nu>2$), random wave model.
\end{remark}
\medskip

We define the {\bf index}  $i(X''(t))$ of a symmetric matrix $M$ as the number of its negative eigenvalues.
We consider the set $\cT =[-T,T]^d$. For $u \in \R \cup \{-\infty\}$, and $k=0,1,\cdots,d$, we define {\bf the number of critical points of $\X$ within $\cT$ with value higher than $u$ and index $k$} by
\begin{equation*}
\crit^k (u,\cT):= \#\big\{ t \in \cT : X'(t)=0, i(X''(t)) =k, X(t)>u\big\}.
\end{equation*}
In order to study the joint distribution of these random variables, we define for $\balpha=(\alpha_0,\alpha_1,\cdots,\alpha_d)\in\R^{d+1}$:
\begin{align} \label{e:Bc}
 \crit^{\balpha} (u,\cT) &:= \sum^d_{k=0} \alpha_k \crit^k (u,\cT);
\notag \\
 \cC^{\balpha}(u,\cT) &:=\frac{ \crit^{\balpha}(u,\cT)-\E( \crit^{\balpha}(u,\cT))}{(2T)^{d/2}}.
\end{align}

\noindent The following theorem is the main result of the paper.
\begin{theorem} \label{t:tcl}
Consider a real-valued random field $\X$ defined on $\R^d$ satisfying Assumptions {\rm(A1)}, {\rm(A2)} and {\rm(A3)}.
Let $u\in\R \cup \{ -\infty\}$, $\balpha\in\R^{d+1}$ and $\cT =[-T,T]^d$ for $T>0$.
Then,
\begin{enumerate}
\item There exists $V^{\balpha}(u)<\infty$ such that
\begin{equation*}
 \lim_{T\to\infty} \Var \big( \cC^{\balpha}(u,\cT) \big) = V^{\balpha}(u).
\end{equation*}

\item The distribution of $\cC^{\balpha}(u,\cT)$ converges, as $T\to\infty$, towards the centered normal distribution with variance $V^{\balpha}(u)$.
\item For $u=-\infty$, setting $\balpha= \sum^d_{k=0}(-1)^k e^d_k$ we have $V^{\balpha}(-\infty)=0$.
\item If in addition we assume {\rm(A4)}, for $u\in\R$ and $\balpha \neq 0$, we have $V^{\balpha}(u)>0$.
\end{enumerate}
\end{theorem}

\begin{remark}
When $u\in\R$, setting $\balpha= \sum^d_{k=0}(-1)^k e^d_k$, we obtain the results by Estrade \& Le\'on \cite{EL} even under more general conditions.
When $u =-\infty$, setting $\balpha= \sum^d_{k=0}e^d_k$, we obtain the result by Nicolaescu \cite{nico}.
\end{remark} 

The following corollary is an immediate consequence of Theorem \ref{t:tcl}.
\begin{corollary}
Under the conditions {\rm(A1)}, {\rm(A2)} and {\rm(A3)}, the random vector $ \left( \cC^0(u,\cT) , \cdots, \cC^d (u,\cT)\right)$ converges, as $T\to \infty$ to a multivariate Gaussian distribution.
When $u= -\infty$, the limit distribution degenerates. When $u\in\R$ and (A4) holds, the limit distribution does not degenerate.
\end{corollary}

\subsubsection*{Sufficient conditions for (A2)}
Our main references for the finiteness of the second moment of the number of critical points are \cite{AL} and \cite{EF}.
As a matter of fact, the proofs of these results contain, as a by-product, the proof of the continuity of the second moment of $N(v,B)$ defined in \eqref{a:1} w.r.t.
$v$. Inspired by the tools of these papers, in   \ref{s:app:geman}, we provide a proof of the following proposition which extends their results and implies {\rm(A2)}.
\begin{proposition}\label{p:prop}
 Suppose that the random field $\X$ satisfies {\rm(A1)} and
 \begin{itemize} 
  \item[] (A5) : For all $\mu \in {\mathbb S}^{d-1}$, the random vector $X''(0) \mu$ has a N.D. distribution.
 \item[] (A6) : The integral $\displaystyle \int \tfrac{ \| r^{(4)}(t)-  r^{(4)}(0)\|} {\|t\|^d} \ \dd t$ converges at $0$.
 \end{itemize}
 Then,  
 \begin{itemize}
  \item For all compact $B \subset \R^d$, the second moment of $N(v,B)$ is finite.
 \item For $B$ sufficiently small, the second moment of $N(v,B)$ is continuous as a function of $v$.
 \end{itemize}
\end{proposition}

\section{Hermite and Wiener Expansions} \label{s:H}
In this section, we give the explicit expression of the chaotic expansion for the numbers of critical points.
We do it in two ways. The first one is to get a Hermite expansion ``\`a la Kratz-Le\'on'' \cite{K:L}.
Secondly, we translate each component in the expansion into a Multiple Wiener-It\^o integral.
The starting point of our analysis is an integral formula for $\crit^k(u,\cT)$.
Set $W(\ve)$ for the volume of a ball of radius $\ve$ in $\R^d$ and $\cA_k$ for the set of symmetric $d\times d$ matrices with index $k$, $k=0,1,\cdots, d$.
\begin{lemma}
Suppose $X(\cdot) $ is a random field satisfying {\rm (A1)} and {\rm (A2)}.
With the above notation, for $k=0,1,\cdots, d$, and $u\in\R$, we have
\begin{subequations}
\begin{equation} \label{e:aprox-cCuk}
\crit^k(u,\cT)
= \lim_{\ve\to 0}  \frac{(-1)^k}{ W(\ve)} \int_{\cT} {\det} (X''(t)) \indicator_{\cA_k}(X''(t)) \indicator_{[u,\infty)} (X(t)) \cdot  \indicator_{ \|X'(t)\|
\leq \ve} \dd t,
\end{equation}
both a.s. and in the $L^2$-sense. When $u=-\infty$, we have
\begin{equation} \label{e:aprox-cCuk2}
\crit^k(-\infty,\cT)
= \lim_{\ve\to 0}  \frac{(-1)^k}{ W(\ve)} \int_{\cT} {\det} (X''(t)) \indicator_{\cA_k}(X''(t))  \cdot  \indicator_{ \|X'(t)\|
\leq \ve} \dd t.
\end{equation}
\end{subequations}
\end{lemma}
Note that in the second case the integrand is a product of two independent factors.
The proof of this lemma follows closely the same lines as that of \citep[Prop.
1.2]{EL}, actually, the only novelty is the introduction of the indicator of ${\cal A}_k$.
\subsection{Hermite expansion}
This section is inspired by \citep{EL}. The Hessian $X''(t)$ can be viewed either as a symmetric $d\times d$ matrix or as a size $\frac12d(d+1)$ vector.
Set $D=\frac12(d+1)(d+2)$ and introduce the $\R^D$-valued stationary Gaussian field $\rX$ by
\begin{subequations}
\begin{equation} \label{e:bx}
\rX(t) = \big( X'(t), X''(t), X(t) \big), t\in\R^d.
\end{equation}
The stationarity of $X$, see (A1), implies that the variance-covariance matrix of $\rX(t)$, $\Xi=\E(\rX(t) \rX(t)^\top)$, does not depend on $t$.
Let $\Lambda$ be such that $\Lambda\Lambda^\top = \Xi$. Note that $\Lambda$ can be chosen, again by the stationarity of $X$, as a block matrix $\Lambda=\Big(\begin{smallmatrix}\Lambda_1 & 0\\ 0& \Lambda_2\end{smallmatrix}\Big)$ due to the independence of $X'(t)$ from $(X''(t),X(t))$.
Thus, we can write
\begin{equation} \label{e:XeY}
\rX(t) = \Lambda \rY(t),
\end{equation}
\end{subequations}
with $\rY(t)$ a standard Gaussian random vector in $\R^D$ for each fixed $t$.
In order to introduce the coefficients of the expansion of $\crit^{\balpha} (u,\cT)$ we exploit the factorization in the integrand in the r.h.s.
of \eqref{e:aprox-cCuk}. For $\underline{y}\in\R^d$ and $\overline{y}=(x,z)\in \R^{\frac12 d(d+1)}\times\R$, set ${\bf{y}}=(\underline{y},\overline{y})\in \R^{D}$ and define 
\begin{equation*}
\tilde{G}_\ve({\bf y}) = G_\ve(\Lambda {\bf y}) = \tilde{\delta}_\ve(\underline{y})\cdot \tilde{f}_k(\overline{y}),
\end{equation*}
with $ \delta_\ve(\cdot) := \frac 1{W( \ve)} \indicator_{ \|\cdot\|
\leq \ve}$, $\tilde{\delta}_\ve = \delta_{\ve}\circ\Lambda_1$, $f_k(x, z) = (-1)^k \det(x) \indicator_{\cA_k}(x) \indicator_{[u,\infty)} (z)$ and $\tilde{f}_k = f_k\circ\Lambda_2$.
Since these functions are in $L^2(\phi_d)$ and $L^2(\phi_{D-d})$ respectively, we can consider their Hermite expansions.
Define the (tensorial) Hermite polynomial, for $p\in\N$ and $\mathbf m=(m_1,\ldots, m_p)\in\N^p$, by $H_{\otimes_{\mathbf m}} := H_{m_1} \otimes \cdots\otimes H_{m_p}$, where $\otimes$ stands for the tensorial product and $H_m$ is the Hermite polynomial of degree $m$.
Fix $k=0,\cdots,d$ and, for ${\bf{n}}=(\underline{n},\overline{n})\in \N^d \times \N^{D-d}$, set
\begin{equation} \label{coeff:a}
   a_{k}({\bf{n}}):= d({\underline{n}}) c(\tilde{f}_{k},{\overline{n}}), 
\end{equation}
where $c(\tilde{f}_{k},{\overline{n}})$ is the $\overline{n}$-th Hermite coefficients of $\tilde{f}_k$, that is:
\[
	c(\tilde{f}_{k},{\overline{n}}) = \frac{1}{\overline{n}!}\int_{\R^{D-d}} 
	\tilde{f}_{k}(\overline{y})H_{\otimes_{\overline{n}}}(\overline{y})\phi_{D-d}	(\overline{y})\dd\overline{y},
\]
and $d({\underline{n}}) = \frac{1}{{\underline{n}}!} \frac{H_{ \otimes{\underline{n}}}({\underline{0}})}{(2\pi )^{d/2}} =\mathop{\lim}\limits_{\ve\to0} c(\tilde{\delta}_\ve,\underline{n})$ with $c(\tilde{\delta}_\ve,\underline{n}) =\frac{1}{{\underline{n}}!}\int_{\R^{d}} \tilde{\delta}_\ve(\underline{y})H_{ \otimes{\underline{n}}}(\underline{y})\phi_d(\underline{y})\dd\underline{y}$ the $ \underline{n}$-th Hermite coefficients of $\tilde{\delta}_\ve$.
\smallskip

The $L^2$ convergence in \eqref{e:aprox-cCuk} implies the following proposition. We still need some notation, for $q\in\N$, let ${\cal J}_q:= \{{\bf{n}}\in \N^D: |{\bf{n}}|=q\}$ and, for $\bf n\in{\cal J}_q$, set $\cH_{{\bf n}}(\cT)= \int_\cT H_{\otimes_{\bf n}} (\rY(t)) \dd t$ and $a_{\balpha} ({\bf n}) = \sum^d_{k=0} \alpha_k a_{k} ({\bf n})$ with $\balpha\in\R^{d+1}$.
\begin{proposition} \label{p:hexp}
Let $X(\cdot) $ be a random field that satisfies {\rm(A1)} and {\rm(A2)}.
Then, with the notation just introduced, for $\balpha\in\R^{d+1}$, in the $L^2(\Omega)$-sense we can expand $\cC^{\balpha} (u,\cT) = \sum_{q=1}^{\infty} {\cC}_q^{\balpha} (u,\cT)$ with
\begin{equation*}
 {\cC}_q^{\balpha} (u,\cT) = \frac{1}{(2T)^{d/2}}\sum_{{\bf{n}}\in {\cal J}_q} a_{\balpha}({\bf{n}}) \cH_{{\bf{n}}} (\cT).
\end{equation*}
That is, ${\cC}_q^{\balpha} (u,\cT)$ is the $q$-th chaotic component of $\cC^{\balpha} (u,\cT)$.
\end{proposition}

\subsection{Multiple It\^o-Wiener Integrals (MWI)}
We use \emph{spectral} stochastic integrals as defined in \cite{major} and \cite[Ch.9]{peccati-taqqu}.
We choose to follow the \emph{isonormal Gaussian process} framework in \cite[Ch.9]{peccati-taqqu}, rather than the explicit MWI construction in \cite{major}.
\subsubsection{Generalities}
Let ${\cal H}$ be the set of complex-valued Hermitian square integrable functions w.r.t.
Lebesgue measure $\dd\nu$ in $\R^d$, that is, ${\cal H}=L^2_H(\R^d,\dd\nu)= \{\psi:\R^d\to\C: \psi(-\nu) =\overline{\psi(\nu)}, \|\psi\|_{{\cal H}}^2=\int_{\R^d}|\psi(\nu)|^2\dd\nu<\infty\}$.
Let $\hB$ be a complex Hermitian Brownian measure on $\R^d$, defined on a probability space $(\Omega,{\cal U},\P)$ such that ${\cal U}$ is generated by $\hB$.
The (real-valued) Wiener integral w.r.t. $\hB$, denoted by 
\begin{equation*}
I^{\hB}_1(\psi) = \int_{\R^d}\psi(\nu) \dd\hB(\nu),
\end{equation*}
is an isometry from ${\cal H}$ into $L^2(\Omega)$, see \cite{major} and \cite[Ch.9]{peccati-taqqu}.
That is, for $\psi_1,\psi_2\in {\cal H}$, we have
\begin{align} \label{e:isometria1}
 \E\Big(I^{\hB}_1(\psi_1)I^{\hB}_1(\psi_2)\Big) &= \int_{\R^d}\psi_1(\nu)\overline{\psi_2(\nu)}\dd\nu = \int_{\R^d}\psi_1(\nu)\psi_2(-\nu)\dd\nu \notag \\
 &= \<\psi_1,\psi_2\>_{{\cal H}}.
\end{align}

The $q$-fold multiple Wiener-It\^o integral w.r.t. $\hB$ is defined, for $\psi_j\in{\cal H},j\leq D$ and ${\bf n}\in{\cal J}_q$, by
\begin{equation*}
 I^{\hB}_q\big(\psi^{\otimes n_1}_{1}\otimes\cdots\otimes \psi^{\otimes n_D}_{D}\big) =\prod^{D}_{j=1} H_{n_j}\Big(I^{\hB}_1(\psi_{j})\Big).
\end{equation*}
Let ${\cal H}_q$ denote the domain of $I^{\hB}_q$, i.e.: the set of those $\psi:(\R^d)^q\to\C$ such that $\psi(-\nu_1,\cdots,-\nu_q) =\overline{\psi(\nu_1,\cdots,\nu_q)}$ and $\|\psi\|^2_{{\cal H}_q} =\int_{\R^{dq}}|
\psi(\nu_1,\cdots,\nu_q)|^2\dd\nu_1\cdots \dd\nu_q$ is finite. The latter integral defines the norm and (thus, also) the inner product in ${\cal H}_q$.
To describe the isometry induced by $I^{\hB}_q$, we introduce the symmetrization of a kernel $\psi\in{\cal H}_q$: 
\begin{equation*} 
 Sym(\psi)(\nu_1,\cdots,\nu_q) := \frac{1}{q!}\sum_{\pi\in {\cal S}_q} \psi(\nu_{\pi(1)},\cdots,\nu_{\pi(q)}),
\end{equation*}
with ${\cal S}_q$ being the group of permutations of $[q]$.
For $\psi\in{\cal H}_q$, we have that $I^{\hB}_q(\psi) = I^{\hB}_q(Sym(\psi))$. Moreover, for $\psi_1\in{\cal H}_q$ and $\psi_2\in{\cal H}_p$, it holds that
\begin{equation} \label{e:isometria}
 \E\big(I^{\hB}_q(\psi_1) I^{\hB}_p(\psi_2)\big) = \delta_{pq} q!
\<Sym(\psi_1),Sym(\psi_2)\>_{{\cal H}_{q}},
\end{equation}
$\delta_{pq}$ being the Kroenecker symbol. Hence, $I^{\hB}_q$ is an isometry from ${\cal H}_q^s:=\{\psi\in {\cal H}_q: \psi=Sym(\psi)\}$, with the modified norm $\sqrt{q!}\|\cdot\|_{{\cal H}_q}$, onto its image, which is called the $q$-th Wiener chaos ${\cal K}_q$.
By convention, ${\cal K}_0=\R$. 

It is well known, see e.g. \citep[Ch.8]{peccati-taqqu} that the (orthogonal) sum $\bigoplus^\infty_{q=0}{\cal K}_q=L^2(\Omega)$.
In other words, if $X\in L^2(\Omega)$, there exists an unique sequence of (symmetric) kernels $\psi_q\in {\cal H}^s_q: q\geq 1$ such that $ X=\sum^\infty_{q=0} I^{\hB}_q(\psi_q)$.
In particular, $I^{\hB}_0(\psi_0)= \psi_0=\E(X)$. The key point is that chaotic r.v.
(i.e: Multiple Wiener-It\^o Integrals) are well suited to study asymptotic normality. 

Now, we restrict and adapt to our framework \citep[Th.
11.8.1]{peccati-taqqu} which states that, in this setting, joint convergence is equivalent to marginal convergence.
Besides, the third item below provides a practical criterion to check it.
For $r\leq p\wedge q$, denote $\bar{\otimes}_r$ the $r$-th contraction operator: $(\phi,\psi)\in{\cal H}^s_p \times {\cal H}^s_q\mapsto \phi\bar{\otimes}_r \psi \in {\cal H}_{p+q-2r}$ defined by
\begin{multline*}
 \phi\bar{\otimes}_r \psi (\nu_1,\cdots,\nu_{p+q-2r})
 =\int_{\R^{dr}} \phi(z_1,\cdots,z_r;\nu_1,\cdots,\nu_{p-r})\\
  \cdot \psi(-z_1,\cdots,-z_r;\nu_{p-r+1},\cdots,\nu_{p+q-2r}) \dd z_1\cdots \dd z_r.
\end{multline*}

\begin{theorem}(\cite{peccati-taqqu}, Th.11.8.1). \label{t:jvsm}
Let $Q\geq 1$. For $q\leq Q$, consider the sequence of kernels $(\psi_{q,n})_{n\geq 1}$ with $\psi_{q,n}\in{\cal H}^s_q$.
Then, as $n\to\infty$ the following conditions are equivalent:
\begin{enumerate}
  \item $\big(I^{\hB}_1(\psi_{1,n}),\cdots,I^{\hB}_{Q}(\psi_{Q,n})\big)$ converges in law towards a centered normal random vector with variance $\diag(\sigma^2_1,\cdots,\sigma^2_Q)$.
 \item For each $q\leq Q$, $I^{\hB}_q(\psi_{q,n})$ converges in law towards a centered normal r.v. with variance $\sigma^2_q$.
 \item For each $q\leq Q$ and $r=1,\cdots,q-1$, the norms of the contractions $\|\psi_{q,n}\bar{\otimes}_r \psi_{q,n}\|_{{\cal H}_q}$ converge to $0$.
 \end{enumerate}
\end{theorem}

\subsubsection{MWI representation}
As said above, in the second step we translate $\cC^{\balpha}_{q}(u,\cT),q\geq 1$ into a Wiener-It\^o integral.
Assumption {\rm (A3)} implies that $\rX (\cdot)$ given by \eqref{e:bx} (and thus $\rY(\cdot)$ defined in \eqref{e:XeY}) has a spectral density.
This fact allows us to find explicit (Hermitian) orthonormal kernels $\widehat{\psi}_{t,j}\in {\cal H}$ such that
\begin{equation} \label{e:Ysometria}
 Y_j(t) = I^{\hB}_1(\widehat{\psi}_{t,j}),\quad j\in[D].
\end{equation}
Indeed, let $f$ stand for the spectral density of $X$, thus, we can write $X(t) = \int_{\R^d}e^{it\cdot\nu}\sqrt{f(\nu)}\dd\hB(\nu)$.
Taking partial derivatives and denoting
\begin{equation*} 
\omega(\nu) := \Big( (i\nu_j)_{1\leq j\leq d}; (-\nu_j\nu_k)_{1\leq j\leq k\leq d};1\Big),
\end{equation*}
we get $\rX(t) = \int_{\R^d}e^{it\cdot\nu} \omega(\nu) \sqrt{f(\nu)}\dd\hB(\nu)$ as a $D$-dimensional spectral representation.
In consequence, using \eqref{e:XeY}, we conclude that,
\begin{equation*}
 \rY(t) = \int_{\R^d}e^{it\cdot\nu}(\Lambda^{-1}\omega(\nu)) \sqrt{f(\nu)}\dd\hB(\nu).
\end{equation*}
Thus, taking coordinates, the kernels in \eqref{e:Ysometria} are given by
\begin{subequations}
\begin{align} \label{e:nucleos-a}
   & \widehat{\psi}_{t,j}(\nu) = e^{it\cdot\nu} \psi_{j}(\nu) \sqrt{f(\nu)},\quad j\in[D],\\
  & {\rm with}\quad \psi_{j}(\nu) = (\Lambda^{-1}\omega(\nu))_j.
\label{e:nucleos-b}
\end{align}
\end{subequations}

\begin{proposition} \label{p:hexp2}
Assume that $\X$ satisfies $(A1)-(A3)$. Then, with the notation in Proposition \ref{p:hexp}, we can write $\cC^{\balpha}_{q}(u,\cT) =I^B_q\big(g^{\balpha}_{q}(u,\cT)\big)$ with
\begin{equation*}
  g^{\balpha}_{q}(u,\cT) = \frac{1}{(2T)^{d/2}} \sum_{{\bf{n}}\in {\cal J}_q} a_{\balpha}({\bf n}) \int_{\cT} \widehat{\psi}^{\otimes n_1}_{t,1}\otimes\cdots\otimes \widehat{\psi}^{\otimes n_D}_{t,D}\dd t.
\end{equation*}
\end{proposition}

Finally, to gain symmetry in the expansion, let us reindex it in the following way.
To each ${\bf n}\in{\cal J}_q$, associate the set of indexes $ {\cal I}_{{\bf n}} :=\{ {\bf m}\in[D]^q: \sum^q_{j=1}\indicator_{\{i\}} (m_j)=n_i, \forall i\leq D\}$ and set 
\begin{equation} \label{coeff:b}
	b_{\balpha}({\bf m}):=\frac{a_{\balpha}({\bf n})}{\# {\cal I}_{{\bf n}}}.
\end{equation}
Note that the family $\{{\cal I}_{{\bf n}}:{\bf n}\in{\cal J}_q\}$ forms a partition of $[D]^q$.
The idea is to replace $\widehat{\psi}^{\otimes n_1}_{t,1}\otimes\cdots\otimes \widehat{\psi}^{\otimes n_D}_{t,D}$ by the equivalent but more symmetric expression $\widehat{\psi}_{t,m_1}\otimes\cdots\otimes \widehat{\psi}_{t,m_D}$ where $n_j$ of the $m_i$'s equal $j$.
Actually, each term in $\bf n$ is replaced by the sum of terms in ${\bf m}\in{\cal I}_{\bf n}$.
For instance, $\widehat{\psi}_{t,1}\otimes \widehat{\psi}^{\otimes 2}_{t,2}$ is replaced by $\widehat{\psi}_{t,1}\otimes \widehat{\psi}_{t,2}\otimes \widehat{\psi}_{t,2} + \widehat{\psi}_{t,2}\otimes \widehat{\psi}_{t,1}\otimes \widehat{\psi}_{t,2} +\widehat{\psi}_{t,2}\otimes \widehat{\psi}_{t,2}\otimes \widehat{\psi}_{t,1}$.
Still, $\widehat{\psi}^{\otimes 3}_{t,1}\otimes \widehat{\psi}^{\otimes 2}_{t,2}$ is replaced by the sum of ten terms which correspond to the ways of allocating three $\widehat{\psi}_{t,1}$ among five factors.
\smallskip

Now, we can state the final version of the $q$-th chaotic component of $\cC^{\balpha}(u,\cT)$.
\begin{proposition} \label{p:hexp3}
Under the hypotheses of Proposition \ref{p:hexp2} and with the notation above, for ${\bf m}=(m_1,\cdots,m_q)$, we have
\begin{equation*}
Sym(g^{\balpha}_{q}(u,\cT)) = \frac{1}{(2T)^{d/2}} \sum_{{\bf{m}}\in [D]^q} b_{\balpha}({\bf m}) \int_{\cT} \widehat{\psi}_{t,m_1}\otimes\cdots\otimes \widehat{\psi}_{t,m_q}\dd t.
\end{equation*}
\end{proposition}

\section{Proof of Theorem \ref{t:tcl}} \label{s:proof}
The proof of Theorem \ref{t:tcl} is divided into several steps.
The finiteness of the limit variance (part 1.) and the CLT (part 2.) follow the same lines as the corresponding proofs in \cite{EL}, so they are presented succinctly.
The remaining parts are treated more carefully.

\subsection{Limit variance}
Here we deal with three different issues: the finiteness of the limit variance in subsection \ref{ss:fv}, its positivity for $u\in\R$ and $\balpha\neq0$ in the isotropic case in subsection \ref{ss:pv} and the singular case $u=-\infty$ in subsection \ref{ss:su}.
There is also a subsection, \ref{ss:ee}, dedicated to obtain an explicit expression of the limit variance which has some interest in its own right.
\subsubsection{Finiteness of the limit variance} \label{ss:fv}
By the orthogonality in $q$, we have
\begin{equation} \label{e:var}
 \Var \big( \cC^{\balpha} (u,\cT) \big) = \sum^{\infty}_{q=1} \sum_{{\bf m},{\bf n}\in{\cal J}_q} a_{\balpha}({\bf{m}}) a_{\balpha}({\bf{n}}) \frac{\E\big[ \cH_{{\bf{m}}} (\cT) \cH_{{\bf{n}}} (\cT)\big]}{(2T)^d}.
\end{equation}

For each fixed $q$, the convergence of the corresponding term follows in the same manner as \citep[Eq. (8)-(11)]{EL}.
Indeed, for fixed $q$, the inner sum has a finite number of terms, thus, it converges if the terms $\frac{1}{(2T)^d}\E\big[ \cH_{{\bf{m}}} (\cT) \cH_{{\bf{n}}} (\cT)\big]$ do.
We have, for fixed $\bf{m},\bf{n}\in{\cal J}_q$
\begin{align*}
\frac{\E\big[ \cH_{{\bf{m}}} (\cT) \cH_{{\bf{n}}} (\cT)\big]}{(2T)^d} 
&= \frac1{(2T)^d} \iint_{\cT^2} \E[H_{\bar{\otimes}_m}(\bY (s)) H_{\bar{\otimes}_n}(\bY (t))] \dd s\dd t\\
&\mathop{\longrightarrow}\limits_{T\to\infty} \int_{\R^d} \E[H_{\bar{\otimes}_m}(\bY (0)) H_{\bar{\otimes}_n}(\bY (t))] \dd t,
\end{align*}
because of stationarity.
Now, Mehler's formula \citep[Eq.(8)]{EL} gives 
\begin{equation*}
 \E\big[H_{\bar{\otimes}_m}(\bY (0)) H_{\bar{\otimes}_n}(\bY (t))\big] \leq \ct \Psi_Y(t)^q, 
\end{equation*}
where $\Psi_{Y}(t)$ is defined as $\Psi(t)$ replacing $r(t)$ by $\Gamma(t) :=\E(Y(0)Y(t))$ in \eqref{e:psi}.
Recall that $\rX(t)=\Lambda\bY(t)$ for $t\in\R^d$ by \eqref{e:XeY}, thus, there exists $K>0$ such that $\Psi_{Y}(t)\leq K\Psi(t)$.
Hence, $\E[H_{\bar{\otimes}_m}(\bY (0)) H_{\bar{\otimes}_n}(\bY (t))] \leq \ct \Psi(t)^q$, and provided that $\Psi\in L^1(\R^d)$ by (A3), we deduce that there exists $V_q^{\balpha}(u)\ge 0$ such that
\begin{equation} \label{d:Vqalpha}
V_q^{\balpha}(u,\cT) := \sum_{{\bf m},{\bf n}\in{\cal J}_q} a_{\balpha}({\bf{m}}) a_{\balpha}({\bf{n}}) \frac{\E\big[ \cH_{{\bf{m}}} (\cT) \cH_{{\bf{n}}} (\cT)\big]}{(2T)^d} \mathop{\to}\limits_{T\to\infty} V_q^{\balpha}(u).
\end{equation}

We now proceed to bound the variance of the tail of the expansion:
\begin{equation} \label{cota:cola}
 \sup_{T} \sum_{q>Q} V_q^{\balpha}(u,\cT) = \sup_{T} \Var \Big( \pi_Q(\cC^{\balpha}(u,\cT))\Big).
\end{equation}
Here, $\pi_Q$ stands for the orthogonal projection onto $\bigoplus_{q>Q}{\cal K}_q$. Note that this fact implies the convergence in \eqref{e:var}.
Set $I_{T}:=[-T,T)^d\cap\Z^d$ and $s[0,1)^d:=s+[0,1)^d$. Thus, 
\begin{equation*}
\cC^{\balpha}(u,\cT) = \frac{1}{(2T)^{d/2}}\sum_{s\in I_{T}} \cC^{\balpha} (u,s[0,1)^d).
\end{equation*}
Using the stationarity of $\rY$, we get
\begin{align*}
V_{T,Q} &:=\Var \big( \pi_Q(\cC^{\balpha}(u,\cT))\big) \\
&= \sum_{s\in I_{T}} \frac{ \# I_T\cap(I_T-s)}{(2T)^d}\ \E\Big(\pi_Q(\cC^{\balpha}(u,[0,1)^d))\cdot \pi_Q(\cC^{\balpha}(u,s[0,1)^d))\Big).
\end{align*}
Now, assumption (A3) gives us the right to choose $a>0$ such that $K\Psi(s)\leq \rho<1$ for $s:\|s\|_{\infty}\geq a$.
Thus, we split $V_{T,Q} = V^1_{T,Q}+ V^2_{T,Q}$ where $V^1_{T,Q},V^2_{T,Q}$ stand respectively for the sums over $s:\|s\|_{\infty}\leq a+1$ and over $s : \|s\|_{\infty} > a+1$.
Consider $V^1_{T,Q}$ first. Clearly there are less than $(2a+2)^d$ terms in the sum and for each one of them, assuming w.l.o.g.
that $2T > a + 1$, we have $\# I_T\cap(I_T-s)\leq (2T)^d$.
Thus, from the Cauchy-Schwarz inequality and stationarity, we have
\[
 |V^1_{T,Q}|
\leq  (2a+2)^d \ \E \Big(\pi_Q(\cC^{\balpha}(u,[0,1)^d))^2 \Big),
\]
which tends to $0$ as $Q\to\infty$ uniformly w.r.t. $T$.
Now, we move to $V^2_{T,Q}$ to find that
\begin{multline*}
\E\Big(\pi_Q(\cC^{\balpha}(u,[0,1)^d))\cdot \pi_Q(\cC^{\balpha}(u,s[0,1)^d))\Big)\\
= \sum_{q>Q} \iint_{([0,1)^d)^2} \E\big[ h_q(\rY(t)) h_q(\rY(s+u))\big]\dd t\dd u,
\end{multline*}
where $h_q(\cdot) = \sum_{{\bf n}\in{\cal J}_q}a_{\balpha}({\bf n})H_{\otimes_{\bf n}}(\rY(\cdot))$.
Arcones' inequality, see Lemma \ref{l:arcones}, implies that
\[
\E\big[ h_q(\rY(t)) h_q(\rY(s+u))\big] \leq K^d\Psi^d(s+u-t) \|h_q\|^2_2.
\]
Since $\|s+u-t\|_\infty\geq a$, we have $K\Psi(s+u-t)\leq\rho$.
Besides, \citep[Lem.4.3]{aadl0} implies that $\|h_q\|^2_2\leq \|\sum^d_{k=0}\alpha_k \tilde{f}_k\|^2_2=\ct<\infty$. Thus, 
\begin{equation*}
|V^2_{T,Q}| \leq \ct \sum_{s\in I_T,\|s\|_{\infty}>a} \sum_{q>Q} \rho^{q-1}\iint_{([0,1)^d)^2} \Psi(s+u-t) \dd tdu \leq \ct \rho^{Q},
\end{equation*}
which tends to $0$ as $Q\to\infty$, uniformly w.r.t.
$T$. This proves that the supremum in \eqref{cota:cola} tends to zero and, thus, that the limit variance $V^{\balpha}(u)$ is finite.
\subsubsection{Explicit expression} \label{ss:ee}
We compute the limit variance $V^{\balpha}_{q}(u)$ of the $q$-th chaotic component $\cC^{\balpha}_{q}(u,\cT) = I^{\hB}_q(g^{\balpha}_{q}(u,\cT))$ as $T\to\infty$, see Proposition \ref{p:hexp3}.
By the isometric property of the MWI \eqref{e:isometria}, we get
\begin{equation*}
V^{\balpha}_{q}(u,\cT) = \E\big( I^{\hB}_q(g^{\balpha}_{q}(u,\cT))^2 \big) = q! \|Sym(g^{\balpha}_{q}(u,\cT))\|^2_{{\cal H}_q}.
\end{equation*}
Recall that $\Gamma_{i,j}(t)=\E[Y_i(t)Y_j(0)]$. Hence, we have
\begin{align*}
V^{\balpha}_{q}(u,\cT) &= \frac{q!}{(2T)^d} \sum_{{\bf m},{\bf m'}\in[D]^q} b_{\balpha}({\bf m}) b_{\balpha}({\bf m}') \iint_{\cT\times\cT} \prod^q_{i=1} \<\widehat{\psi}_{t,m_{i}}, \widehat{\psi}_{s,m'_{i}}\>_{{\cal H}} \dd s\dd t\\
&= \frac{q!}{(2T)^d} \sum_{{\bf m},{\bf m}'\in[D]^q} b_{\balpha}({\bf m}) b_{\balpha}({\bf m}') \iint_{\cT\times\cT} \prod^q_{i=1} \Gamma_{m_{i},m'_{i}} (t-s) \dd s\dd t\\
&\mathop{\to}\limits_{T\to\infty} q!
\sum_{{\bf m},{\bf m}'\in[D]^q} b_{\balpha}({\bf m}) b_{\balpha}({\bf m}') \int_{\R^d} \prod^q_{i=1} \Gamma_{m_{i},m'_{i}}  (t) \dd t =V^{\balpha}_{q}(u).
\end{align*} 
In the second line, we used the isometry \eqref{e:isometria1} of the simple stochastic integrals \eqref{e:Ysometria} and we replaced $\pi^{-1}$ by $\pi$ to simplify the notation.
Now, from \eqref{e:Ysometria}, \eqref{e:nucleos-a} and \eqref{e:nucleos-b}, we have that $\Gamma_{i,j}(t)={\cal F}(\psi_i\psi_jf) (t)$ with $\psi_j(\cdot)=(\Lambda^{-1}\omega(\cdot))_j$.
As the product of covariances corresponds, via the Fourier transform, to the convolution of the spectral densities, we have
\begin{equation*}
 \prod^q_{i=1} \Gamma_{m_{i},m'_{i}}  (t) = {\cal F}\Big( \mathop{\ast}^q_{i=1} \psi_{m_{i}}\psi_{m'_{i}}f \Big) (t).
\end{equation*}
Thus, since the integral on the whole space of a function coincides with its Fourier transform evaluated at zero, we deduce that
\begin{align*}
V^{\balpha}_{q}(u) &= q!\sum_{{\bf m},{\bf m}'\in[D]^q} b_{\balpha}({\bf m}) b_{\balpha}({\bf m}') \int_{\R^d} {\cal F}\Big( \mathop{\ast}^q_{i=1} \psi_{m_{i}} \psi_{m'_{i}}f \Big) (t) \dd t\\
&= q!\sum_{{\bf m},{\bf m}'\in[D]^q} b_{\balpha}({\bf m}) b_{\balpha}({\bf m}') {\cal F}\Big({\cal F}\Big( \mathop{\ast}^q_{i=1} \psi_{m_{i}} \psi_{m'_{i}}f \Big)\Big) (0).
\end{align*}
Now, using the Fourier inversion formula, we get 
\begin{equation*}
V^{\balpha}_{q}(u) = q!(2\pi)^d\sum_{{\bf m},{\bf m}'\in[D]^q} b_{\balpha}({\bf m}) b_{\balpha}({\bf m}') \mathop{\ast}^q_{i=1} \psi_{m_{i}} \psi_{m'_{i}}f (0).
\end{equation*}
The convolution above can be written as
$ \int_{{\cal W}} \prod^q_{i=1} (\psi_{m_{i}} \psi_{m'_{i}} f)(\nu_i) \dd\sigma(\nu_1,\cdots \nu_{q})$,
with ${\cal W}:=\{(\nu_1,\cdots,\nu_q)\in\R^{dq}:\nu_1+\cdots+\nu_q=0\}$ and $\dd\sigma(\cdot)$ the Lebesgue measure on ${\cal W}$.
Finally, we arrange the terms to
\begin{subequations}
\begin{equation} \label{e:Vqfinal}
V^{\balpha}_{q}(u) = q!(2\pi)^d \int_{{\cal W}} P_q(\nu_1,\cdots,\nu_q)^2 {\mathbf f}_q(\nu_1,\cdots,\nu_q) \dd\sigma(\nu_1 \cdots \nu_{q}),
\end{equation}
with
\begin{align}
P_q(\nu_1, \cdots, \nu_{q}) & = \sum_{{\bf m}\in[D]^q} b_{\balpha}({\bf m}) \prod^q_{i=1} \psi_{m_{i}} (\nu_i), \label{e:forme1}\\
{\mathbf f}_q(\nu_1,\cdots,\nu_q) &= \prod^q_{i=1} f (\nu_i).
\label{e:Fq}
\end{align}
\end{subequations}
See the analogous expression in \citep[Cor. 1.3]{Gass-esp}. Needless to say that the expression for $P_q$ is heavy combinatorially spealing.
\subsubsection{Positivity of the limit variance \texorpdfstring{$u\in\R$}{}} \label{ss:pv}
In this subsection we assume that the random field $\X$ satisfies Assumptions (A1) to (A4).
The last Assumption, (A4), ensures the positivity of ${\mathbf f}_q$ close to the origin and allows to manipulate some Hermite coefficients thanks to the simplifications implied by the isotropy condition.
Our aim is to prove that for $u\in\R$, the only way to get a degenerate limit law is that the coefficients $\alpha_0,\cdots,\alpha_d$ vanish.
Hence, assume that the limit variance $V^{\balpha}(u)$ vanishes, that is, assume that $V^{\balpha}_{q}(u) =0$ for every $q\geq1$, see \eqref{d:Vqalpha}.
This fact provides us an infinite number of conditions on the coefficients $\alpha_k:k=0,1,\cdots,d$ but, even in this way, it is not direct to conclude that they vanish, see the related discussion in \cite[S.1.1.2]{Gass-esp}.
Our way out is to focus on a very particular and convenient class of Hermite coefficients which we are able to isolate and manipulate in order to get some conditions related to the eigenvalues of a GOE random matrix.
\medskip

The next lemma, which is proved in   \ref{a:pqcero}, provides us the first step by reducing the problem to one kind of Hermite coefficients.
\begin{lemma} \label{l:pqcero}
Assume (A1) to (A4). Assume further that $V^{\balpha}(u)=0$. Hence, for each $q\ge 1$, we have that
\begin{equation*}
 \sum_{k=0}^d \alpha_kb_{k}(D\u_q) = 0,
\end{equation*}
where $\u_q$ stands for the "all ones" vector in $[D]^q$.
\end{lemma}
So we are led to study the Hermite coefficients $b_k(D\u_q):k=0,1,\cdots,d$; $q=1,\cdots$.
From \eqref{coeff:b} and from \eqref{coeff:a} we get $b_{k}(D\u_q) =  a_k(qe^D_D)=d(0)c(\tilde{f}_k,qe^{D-d}_{D-d}):k=0,\cdots,d$; $q=1,\cdots$.
Hence, omitting $t$, since $\tilde{f}_k(Y'',Y)=f_k(X'',X)$, we have for each $q\geq 1$
\begin{equation*} 
 0=q!c(\tilde{f}_k,qe^{D-d}_{D-d}) = \E\left[\sum_{k=0}^d(-1)^k\alpha_k \det(X'')\indicator_{A_k}(X'')\indicator_{[u,\infty)}(X)H_q(Y_D)\right].
\end{equation*}
This expression has the drawback of involving both the original random vector $\rX$ and its standardized version $\rY$.
Next lemma, which proof can be found in   \ref{a:somel}, transforms this linear system of equations into a more tractable one.
Set $c_d=\sqrt{2\lambda_4/3}$ and define, for $k=0,\cdots,d$, the functions
\begin{equation} \label{d:Fk}
 F_k(z) := c_d^{-d}\ \E\big[(-1)^k\det(X'')\indicator_{{\mathcal A}_k} (X'')\mid {\rm tr}(X'')= -d z\big], z\in\R.
\end{equation}
\begin{lemma} \label{l:traza}
On the conditions of Lemma \ref{l:pqcero}, for $z\in\R$, $ \sum^d_{k=0}\alpha_k F_k(z) = 0$.
\end{lemma}
Hence, it suffices to prove that the functions $F_k:k=0,\cdots,d$ are linearly independent. We can further simplify the expression for $F_k:k=0,\cdots,d$.
The conditional distribution of $X''$ given that ${\rm tr}(X'') = -d z$ equals the distribution of $c_d\big(G_d - \frac{1}{d}{\rm tr}(G_d)\I_d - z\I_d\big)$, where $G_d$ is a $d\times d$ GOE random matrix, see   \ref{a:pqcero}.
Besides, since each function in the expectation in \eqref{d:Fk} depends on $X''$ only through its eigenvalues, we can rewrite $F_k$ in terms of the (well known) distribution of the eigenvalues of a GOE random matrix.
\begin{equation*}
 F_k(z) = \E\Bigg[(-1)^k\prod^d_{i=1}(\nu_i - \bar{\nu}-z) \cdot \indicator_{\nu_1-\bar{\nu}<\cdots<\nu_k-\bar{\nu}<z<\nu_{k+1}-\bar{\nu}<\cdots<\nu_d-\bar{\nu}}\Bigg],
\end{equation*}
where $\bar{\nu}=\sum^d_{i=1}\nu_i/d$. Furthermore, to simplify the computations, we define
\begin{equation*}
 G_k(z) := \E \Big[(-1)^k\prod^d_{i=1}(\nu_i-z) \cdot \indicator_{\nu_1<\cdots<\nu_k<z<\nu_{k+1}<\cdots<\nu_d}\Big].
\end{equation*}
We can reduce the problem to the linear independence of the $G_k$'s.
Indeed, for a GOE random matrix, the centered eigenvalues $\mu_i:\nu_i-\bar{\nu}:i=1,\cdots,d$ are independent from their mean $\bar{\nu}$.
Thus, conditioning on the mean, we get $G_k(z) = \int^\infty_{-\infty} F_k(\ell+z) \phi_{1/d}(\ell) \dd\ell$, with $\phi_{1/d}$ standing for the density of a centered normal random variable with variance $1/d$.
Clearly, any vanishing non trivial combination of the $F_k$'s gives such a combination for the $G_k$'s.
Hence, it suffices to state the linear independence of the $G_k$'s.
This is a direct consequence of the next lemma, which proof can be found in   \ref{a:somel}.
\begin{lemma} \label{l:LMethod}
 With the above notation. As $z\to\infty$, there exist constants $D^d_k$ and positive constants $C^d_k$ such that
\begin{equation*}
 G_k(z) \simeq C^d_k z^{D^d_k} \exp{\left\{-\frac{d-k}{2}z^2\right\}}.
\end{equation*}
In particular, for $d>1$, the functions $G_k:k=0,\cdots, d$ are linearly independent.
\end{lemma}
This concludes the proof of the positivity of the limit variance $V^{\balpha}(u)$ for $u\in\R$ and $\balpha\neq0$.
\subsubsection{Singularity in the case \texorpdfstring{$u=-\infty$}{}} \label{ss:su}
Let $\chi (\cT)$ be the Euler characteristic of $\cT$ which takes, of course, the value 1. To specify the relation of $\chi (\cT)$ with the modified Euler characteristic $\Phi(\cT)$, we study the boundary of $\cT$.
For $L \subset [d]$ set $\ell=\#L$ and ${\bar{L}}=[d]\setminus L$. A face of dimension $ d-\ell$ of $\cT$ is indexed by
\begin{itemize}
 \item the set $L$ of coordinates that are fixed to $-T$ or $T$,
 \item the set $ \{i_1, \ldots,i_\ell\}$ indicating for each fixed coordinate if it is fixed to $-T $ ($i_j=-1$) or to $T$ ($i_j=1$).
\end{itemize}
Such a face is denoted by $\cT_{L, i_1, \ldots, i_\ell}$ with the special case $\cT_{\emptyset } =\cT $.
Define for $k=0,\ldots,d-\ell$,
\begin{multline*}
  \mu_k(L,i_1,\ldots, i_\ell)  \\ := 
   \#\Big\{v\in \cT_{L,i_1,\ldots, i_\ell}  : X'_j(v) =0, \mbox{ for }j \in {\bar{L}};
i (X''_{ij}(v), i,j \in {\bar{L}}) =d-\ell-k;  \\\mbox{ the $\ell$ outward derivatives are positive}\Big\}.
\end{multline*}
We have \cite[Sec.
9.4]{adlertaylor}
\begin{equation*} 
\chi(\cT ) = \sum_{ L\subset [d] } \sum_{ \{i_1, \ldots,i_\ell\} \in \{-1,1\}^\ell} \sum^{d-\ell}_{k=0}(-1)^k\mu_k(L,i_1,\ldots, i_\ell).
\end{equation*}
In the sum above the term corresponding to $L=\emptyset$ is the modified Euler  characteristic $\Phi(\cT) $:
\[
\Phi(\cT) := \sum_{k=0}^d (-1)^{k}\crit^{d-k} (-\infty,\cT).
\]
Theorem \ref{t:tcl}, applied to a face $\cT_{L, i_1, \ldots, i_\ell}$ with $L\neq \emptyset$, implies that $$ \Var(\mu_k(L,i_1,\ldots, i_\ell))\leq \ct T^{d-\ell}.$$ This proves that the contribution of this term to the variance of $ T^{-\frac{d}{2}} \chi(\cT )$ is negligible.
This gives that
\[
  \Var \Big(T^{-\frac{d}{2}}\big( \Phi(\cT)- \chi(\cT)\big) \Big) \to 0
\]
Since $\chi(\cT) =1$,  we get $\Var \big(T^{-\frac{d}{2}} \Phi(\cT) \big)\to 0$ proving the singularity.
\subsection{Central limit theorem} 
From \eqref{cota:cola}, we have $\lim_{Q\to\infty} \sup_{T} \pi_Q \big( \cC^{\balpha}(u,\cT)\big) = 0$ in the $L^2$-sense.
Hence, it suffices to fix $Q$ and to establish the CLT for the partial sums $\sum^Q_{q=1}\cC^{\balpha}_{q}(u,\cT)$ which, following Theorem \ref{t:jvsm}, is equivalent to state the CLT for each fixed chaotic component $\cC^{\balpha}_{q}(u,\cT) = I^{\widehat B}_q(g^{\balpha}_{q}(u,\cT))$, $q\leq Q$.
Recall that
\begin{equation*}
g^{\balpha}_{q}(u,\cT) = \frac{1}{(2T)^{d/2}} \sum_{{\bf m}\in[D]^q} b_{\balpha}({\bf m}) \int_{\cT} \widehat{\psi}_{t,m_1}\otimes\cdots\otimes \widehat{\psi}_{t,m_q} \dd t.
\end{equation*}
and that $\Gamma_{i,j} (t-s)=\E(Y_i(t)Y_j(s)) =\<\widehat{\psi}_{t,i}, \widehat{\psi}_{s,j}\>_{{\cal H}}$, see Proposition \ref{p:hexp3}.
According to Theorem \ref{t:jvsm}, it suffices to prove that the norm of the contractions $g^{\balpha}_{q}(u,\cT)\bar{\otimes}_{r} g^{\balpha}_{q}(u,\cT)$, $r\in[q-1]$ tend to zero as $T\to\infty$, see \cite[p.17]{EL} for an analogous computation.
Since $q$ is fixed, by subadditivity, the coefficients play no role in this convergence.
Thus, let us concentrate on the norm of the contraction of two normalized integrals as those in $g^{\balpha}_{q}(u,\cT)$.
\begin{multline*}
 \Big[\frac{1}{(2T)^{d/2}}\int_{\cT} \widehat{\psi}_{t,m_1}\otimes\cdots\otimes \widehat{\psi}_{t,m_q} \dd t\Big] \bar{\otimes}_r
 \Big[\frac{1}{(2T)^{d/2}}\int_{\cT} \widehat{\psi}_{s,n_1}\otimes\cdots\otimes \widehat{\psi}_{s,n_q} \dd s\Big]\\
 =\frac{1}{(2T)^d}\int_{(\cT)^2} \prod^{r}_{k=1}\Gamma_{m_kn_k}(t-s)\cdot \bigotimes^{q}_{k=r+1}\widehat{\psi}_{t,m_k}\otimes\widehat{\psi}_{s,n_k}\dd s \dd t.
\end{multline*} 
Now, taking the square of the norm we get
\begin{multline*}
\frac{1}{(2T)^{2d}}\int_{(\cT)^4} \prod^{r}_{k=1}\Gamma_{m_kn_k}(t-s) \prod^{r}_{k=1}\Gamma_{m_kn_k}(t'-s') \prod^{q}_{k=r+1}\Gamma_{m_km_k}(t'-t)\\
\cdot \prod^{q}_{k=r+1}\Gamma_{n_kn_k}(s-s')\ \dd s\dd t\dd s'\dd t'\\
 \leq \frac{1}{(2T)^{2d}}\int_{(\cT)^4} \Psi(t-s)^r  \Psi(t'-s')^r \Psi(t'-t)^{q-r} \Psi(s-s')^{q-r}\ \dd s\dd t\dd s'\dd t'\\
 \leq \frac{\ct}{(2T)^{d}}\int_{(\R^d)^3} \Psi(v_1)^r  \Psi(v_2)^r  \Psi(v_3)^{q-r}\ \dd v_1\dd v_2\dd v_3.
\end{multline*}
To get the last bound, we used an isometric change of variables and the fact that $\Psi$ is bounded.
Afterwards, we enlarged the domain $\cT$ to the whole space $\R^d$.
Assumption {\rm (A3)} easily shows that the latter expression is bounded by $\ct/(2T)^d$. The CLT follows.


\begin{appendix}\label{app}

\section{Wick's formula}\label{appA}
Recall Wick's formula, for $X_1,X_2,X_3,X_4$ centered and jointly Gaussian,
\begin{align*}
\E[X_1X_2X_3X_4] &= \E[X_1X_2]\E[X_3X_4] + \E[X_1X_3]\E[X_2X_4] + \E[X_1X_4]\E[X_2X_3],\\
\E[H_2(X_1)X_2X_3] &= 2\E[X_1X_2]\E[X_1X_3].
\end{align*}
This is a consequence of the well known diagram formula, see \cite[L.A.1]{nico}.
\section{Arcones' inequality} \label{appB}
Let $W$ be a standard Gaussian vector on $\R^n$ and $h:\R^n\to\R$ a measurable function such that $\E[h(W)]=0$ and $\E[h^2(W)]<\infty$.
The Hermite rank of $h$ is defined as
\[
\mbox{rank}(h)=\inf\Big\{\tau:\, \exists\; \mathbf k\in\N^n\,, |\mathbf k|=\tau\,; \E[h(W) H_{\otimes\mathbf k}(W)]\neq0\Big\}.
\]
Then, we have the classical result.
\begin{lemma}[\cite{arcones}, pp. 2245] \label{l:arcones}
Let $W=(W_1,\cdots,W_n)$ and $Z=(Z_1,\cdots,Z_n)$ be two mean-zero Gaussian random vectors on $\R^n$ such that $\mathbb E[W_jW_k]=\mathbb E[Z_jZ_k]=\delta_{j,k},$ for $1\le j,k\le n$.
Set, for $1\le j,k\le n$, $r^{(j,k)}=\mathbb E[W_jZ_k]$. Let $h$ be a function on $\R^n$ with Hermite rank $\tau\geq 1$.
Define the Arcones' coefficient by
\begin{equation*}
\Psi := \max\Bigg\{\max_{1\le j\le n}\sum_{k=1}^n|r^{(j,k)}|,\max_{1\le k\le n}\sum_{j=1}^n|r^{(j,k)}|\Bigg\}.
\end{equation*}
Then,
\begin{equation*}
|\Cov(h(W),h(Z))|\le \mathbb E[h^2(W)] \Psi^{\tau}.
\end{equation*}
\end{lemma}

\section{Proof of Proposition \ref{p:prop}} \label{s:app:geman}
Let $B$ be a Borel set of $\R^d$ and let $N(v,B)$ be as in \eqref{a:1}.
By use of the triangle inequality, we see that it suffices to establish the finiteness of the second moment for a sufficient small set $B$ and that we can limit our attention to the case where $B$ is the ball $B_r$ with center $0$ and radius $r$ with $r$ sufficiently small.
Our main tool is the order-$2$ Kac-Rice formula which requires the distribution of $(X'(s),X'(t))$ to be N.D. for all $ s,t \in B_r$, $s\neq t$.
This will be proved later on in \eqref{e:density}. Then, using \cite[Th.4]{aal}, we have
\begin{multline}\label{e:krf:2}
  \E\big(N(v,B_r) \big(N(v,B_r)-1\big)\big) \\ =
  \int_{(B_r)^2} \E_{\mathcal C} \big|\det(X''(s)) \det(X''(t))\big|
p_{X'(s),X'(t)}(v,v) \dd s\dd t.
\end{multline}
Here, $\E_{\mathcal C}$ is the expectation conditional to $\mathcal C := \{ X'(s) = X'(t) =v\}$.
Because of stationarity, the integral in the r.h.s. of \eqref{e:krf:2} can be bounded by
\begin{equation}\label{e:krf:3}
 \ct   \int_{B_{2r}} \E_{\mathcal C}  \big|\det(X''(0) \det(X''(t)\big| p_{X'(0),X'(t)}(v,v) \dd t.
\end{equation}
We can now study the validity of the order-$2$ Kac-Rice formula. To this end, we study the joint density
\[
  p_{X(0),X(t)} (0,0) = \ct \big( \det \Var(X(0),X(t))\big) ^{-\frac12} .
\]
We use a sequence $t_i$, $i=1,2,\cdots$ such that $\frac{t_i}{\|t_i\|} \to \mu \in {\mathbb S}^{d-1}$. Using the fact that a determinant is invariant by adding to some row (or column) a linear combination of the others rows (or columns) \cite{belyaev1966}, using (A1) we get
\begin{multline*}
  \det \big(\Var(X(0),X(t_i))\big) = 
 \|t_i\|^{2d}\det \big(\Var(X(0), \frac{1}{\|t_i\|} \big(X(t_i)-X(0)\big)\big)\\
  \simeq \|t_i\|^{2d}\det \big(\Var(X(0),X_\mu'(0))\big).
\end{multline*}
Note that, because of the independence of $X(0)$ from $X'(0)$, the limit above is bounded uniformly in $\mu$. By an argument of compactness, this implies that, as $t\to 0$ in any manner, in the neighbourhood of $t=0$, we have
\begin{equation}\label{e:density}
   p_{X(0),X(t)} (0,0)   \leq \ct \|t\|^{-d} .
\end{equation}
This implies in turn that, for $t$ sufficiently small, the joint density is N.D. and that we can apply the order-$2$ Kac-Rice formula as soon as the considered set $B$ is sufficiently small.
Now, we can write the integral in \eqref{e:krf:3} in polar coordinates setting $t= \rho \mu$ for $\rho >0$ and $\mu \in {\mathbb S}^{d-1}$.
We perform the integral in that order
\[
  \int_{{\mathbb S}^{d-1}} \dd\sigma(\mu) \int_0 ^{2r} \rho^d\, \E_{\mathcal C} \big|\det(X''(0) \det(X''(\rho \mu )\big| p_{X'(0),X'(\rho \mu)}(v,v) \dd\rho.
\]
We see that a uniform bound (w.r.t. $\mu$) of the radial limits, as $\rho \to 0$, is sufficient to conclude.

Consider the conditional expectation
\begin{equation*} 
 \mathcal A(\rho \mu ,v) :=\E_{\mathcal C}\big|\det(X''(0))\det(X''(\rho \mu))\big|.
\end{equation*}
Note that $\rho$, $\mu$ and $v$ appear in the condition $ \mathcal C$. Using the Cauchy-Schwarz inequality and the symmetry of the roles of $0$ and $\rho \mu$, we have
\[
   \mathcal A(\rho \mu ,v) \leq \E_\mathcal C \big({\rm det}^2(X''(0))\big).
\]
Condition $\mathcal C$ can be written as $ 
\mathcal C = \{X'(0) =v, \frac{X'(\rho\mu) -X'(0)}{\rho} =0\}$. The $2d\times 2d$ variance-covariance matrix of the conditioners defining $\mathcal C$ is now
\begin{equation}\label{e:var1}
  \Var :=\left(\begin{array}{cc} \Lambda_{1}& - \rho K(\rho, \mu) \\ - \rho    K(\rho, \mu)   &  2 K(\rho, \mu) \end{array}\right),
\end{equation}
where
\begin{equation}\label{e:var2}
 2  K(\rho, \mu) = 2\rho^{-2} \big(\Lambda_1 + r''(\rho\mu)\big)  = r^{(4)}(0)_{:,:,\mu,\mu}+ o (1),
\end{equation}
being $r^{(4)}(0)_{:,:,\mu,\mu}$ the matrix whose $(i,j)$-entry is $r^{(4)}(0)[e_i,e_j,\mu,\mu]$, $i,j\in[d]$.
The covariance between $X''(0)$ and $X'(t)$ is given by the tri-variant $d^3$-tensor $r^{(3)}(t)$.
This implies that the covariance between $X''(0)$ and the conditioners takes the value
\[
\Cov := \Cov \Big(X''(0),\big( X'(0) ,  \frac{X'(\rho\mu) -X'(0)}{\rho} \big)\Big) = \Big(0_{d^3}, \frac{1}{\rho}\big(r^{(3)} (\rho\mu)   \big) \Big).
\]

We now consider another conditioning, namely with respect to $\mathcal C':=\{X'(0) = v, X''(0) \mu =0\}$.
It is trivial to see that $\E_{\mathcal C'} ({ \rm det}^2(X''(0))) =0$.
The variance-covariance matrix of the new conditioners is
\[
  \Var ':=\left(\begin{array}{cc} \Lambda_1 & 0 \\ 0  &  r^{(4)}(0)_{:,:,\mu,\mu}\end{array}\right).
\]
Because of (A5), this matrix is N.D. The covariance is now
\[
\Cov' := \Cov \big(X''(0),\big( X'(0),  X''(0) \mu \big)\big) = \Big(0_{d^3},  r^{(4)}(0)_{:,:,:,\mu}\Big),
\]
where $r^{(4)}(0)_{:,:,:,\mu}$ is the $3$-tensor with $(i,j,k)$-entry equal to $r^{(4)}(0)[e_i, e_j,e_k,\mu]$, $i,j,k\in[d]$.
The regression formulas imply that
\begin{align*}
   \E_{\mathcal C} (X''(0))  &= \Cov \  \Var^{-1} \left( \begin{array}{c} v 1_{d} \\ 0_d\end{array}\right) ,\\
    \Var_{\mathcal C} (X''(0))  &= r^{(4)} (0)- \Cov  \ \Var^{-1}  \Cov ^\top, \\
  \E_{\mathcal C'} (X''(0))  &= \Cov' (\Var')^{-1} \left( \begin{array}{c} v 1_{d} \\ 0_d\end{array}\right) , \\
    \Var_{\mathcal C'} (X''(0))  &= r^{(4)} (0)- \Cov' (\Var')^{-1}  (\Cov') ^\top,
\end{align*}
where $1_d$ is the ``all one vector'' of size $d$.
These formulas need some additional details: for the first one, for instance, let $C_{ijk}$ be the entries of $\Cov$ and $V_{k,k'}$ those of $ \Var^{-1}$, then, the $(i,j)$-entry of $\E_{\mathcal C} (X''(0))$ is $ v \sum^{2d}_{k=0}\sum^d_{k'=0} C_{ijk} V_{kk'}$.
For the second one, the $(i_1,i_2,i_3,i_4)$-entry of $\Cov\ \Var^{-1} \Cov ^\top $ is given by $\sum_{k,k'} C_{i_1,i_2,k} V_{k,k'} C_{i_3,i_4,k'}$.
In these formulas, the only point that can cause divergence is the inversion of the variance-covariance matrices.
For $ \Var'$, it depends on $\mu$ only and, by (A5), it is uniformly, in $\mu$, N.D. By \eqref{e:var1} and \eqref{e:var2}, $\Var$ converges to $\Var'$ again uniformly by a compactness argument.
As a consequence, the inversion of the variance-covariance matrices does not present problems.

We use the following lemma.
Its proof is given at the end of this section.
\begin{lemma} \label{l:lip}
Let $M$ be a $d\times d $ Gaussian matrix with variance-covariance $4$-tensor $T$ and expectation matrix $E$, then $\E\big(\det^2(M)\big)$ is a locally Lipschitz function of $T$ and $E$.
\end{lemma}
The remarks above imply that the conditional expectations and variances $\E_{\mathcal C} ({ \rm det}^2(X''(0)))$, $\E_{\mathcal C'} ({ \rm det}^2(X''(0)))$, $\Var_\mathcal C (X''(0))$ and $\Var_{\mathcal C'}(X''(0))$ are bounded uniformly in $\mu$.
This implies that the locally Lipschitz function $\E\big(\det^2(M)\big)$ is in fact globally Lipschitz, giving
\begin{multline*}
   \E_{\mathcal C} ({ \rm det}^2(X''(0))) = \E_{\mathcal C} ({ \rm det}^2(X''(0)))  - \E_{\mathcal C'} ({ \rm det}^2(X''(0)))\\
    \leq \ct \|
\Var_\mathcal C (X''(0)) - \Var_{\mathcal C' }(X''(0))\| +\ct  \| \E_\mathcal C (X''(0)) - \E_{\mathcal C'} (X''(0))\|.
\end{multline*}
Hence, we have
\begin{equation} \label{e:ec}
    \E_{\mathcal C} ({ \rm det}^2(X''(0))) \leq \ct \| \Var-\Var' \| + \ct \|
\Cov -\Cov' \|.
\end{equation}
Let us consider the first term in the r.h.s. of \eqref{e:ec}.
It is bounded by
\[
  \big[ 2 \rho^{-2} \big(r''(\rho\mu) - r''(0)\big ) \big]  - r^{(4)}(0)_{:,:,\mu,\mu} .
\]
By an application of the Taylor formula with integral remainder, this last term equals
\[
  \int_0^1  W(\xi) \big[ r^{(4)}(\xi \mu \rho)_{:,:,\mu,\mu}-r^{(4)}(0)_{:,:,\mu,\mu} \big] \dd\xi.
\]
where $ W(\xi)$ is some weight function. Let us consider the second term.
It is bounded by
\[
\ct \Big( \frac{1}{\rho}\big(r^{(3)} (\rho\mu) \big) - r^{(4)}(0)_{:,:,:,\mu} \Big).
\]
By direct integration,
\begin{multline*}
  \frac{1}{\rho}\big(r^{(3)} (\mu \rho)   \big)  -  r^{(4)}(0)_{:,:,:,\mu} =
   \frac{1}{\rho}\big(r^{(3)} (\mu \rho)  - r^{(3)} (0) \big)  -  r^{(4)}(0)_{:,:,:,\mu}\\
   = \int_0^1 r^{(4)}(\xi \mu \rho)_{:,:,:,\mu} -r^{(4)}(0)_{:,:,:,\mu} \  \dd\xi.
\end{multline*}
Giving that the r.h.s. of \eqref{e:ec} is bounded by
\begin{multline*}
  \ct \int_0^1  W' (\xi) \|r^{(4)}(\xi \mu \rho)_{:,:,:,\mu} -r^{(4)}(0)_{:,:,:,\mu}\|
\dd\xi\\
  \leq \ct \int_0^1  \|r^{(4)}(\xi \mu \rho)_{:,:,:,\mu} -r^{(4)}(0)_{:,:,:,\mu}\| \dd\xi,
\end{multline*}
since $W'(\xi)$ is bounded.
Gathering together the different bounds obtained so far, we get that the r.h.s.
in \eqref{e:krf:3} is bounded by
\begin{multline*}
 \ct \int_{B_{2r}}\int^1_0 \frac{\|r^{(4)}(\xi t)-r^{(4)}(0)\|}{\|t\|^d} \dd\xi \dd t + O (\|t\|^{d-1})\\
= \ct \int^1_0 \int_{B_{2r}} \frac{\|r^{(4)}(\xi t)-r^{(4)}(0)\|}{\|t\|^d} \dd t \dd\xi +  O (\|t\|^{d-1}).
\end{multline*}
This concludes the proof of the first assertion.

The same calculation proves the uniform domination of the integrand in the Kac-Rice formula of order 2, this implies the continuity.
\subsubsection*{Proof of Lemma \ref{l:lip}}

Note that $\E\big(\det^2(M)\big) $ is the expectation of a polynomial of degree $2d$ in the entries of $M$.
To compute it, we can decompose each entry of $M$ as the sum of its expectation and a centered random variable.
It is well known that the expectation of the product of $m$ Gaussian centered variables (possibly equal) is zero if $m$ is odd and is given by the Wick formula if $m$ is even as the sum over pairwise groupings of products of the $m/2$ covariances associated at the grouping.
In conclusion, $\E\big(\det^2(M)\big)$ is a degree $2d$ polynomial in the $d(d+1)/2$ entries of the expectation and in the $d(d+1)(d(d+1) +2)/8$ entries of the variance-covariance matrix.
But these details are not really needed: all that we need to know is that this function is a polynomial and hence, a local Lipschitz function.
\section{The distribution of \texorpdfstring{${\bf X}''(t)$}{} under isotropy} \label{a:pqcero}

We omit $t$ for brevity. Assume that $X$ is isotropic.
Lemma 2.3 in \cite{AD2022} computes the variance-covariance matrix of $\rX$.
In particular, Lemma 3.1 expresses the distribution of the Hessian matrix $X''$ in terms of the Gaussian Orthogonal Ensemble (GOE) in the following manner: let $G_d$ be a size $d$ GOE random matrix and $Z$ an independent centered normal random variable with variance $1/2$.
Then $X''(t)$ is equal in distribution to
\begin{equation*} 
 \sqrt{\frac{2\lambda_4}{3}}(G_d - Z \I_d).
\end{equation*}
The following facts follow directly from Lemma 2.3 in \cite{AD2022}.
Under the condition ${\rm tr}(X'') = -d z$, $X''$ is distributed as
\begin{equation*} 
  \sqrt{\frac{2\lambda_4}{3}}\Big(G_d - \Big[\frac{1}{d}{\rm tr}(G_d)-z\Big]\I_d \Big).
\end{equation*}
where $G_d$ is a GOE random matrix.
\smallskip

In the rest of this section we describe a convenient form for the square root $\Lambda$ of the variance-covariance matrix $\Xi$ of $\rX(t)$ defined in \eqref{e:bx}.
Recall that $\Lambda$ is a block-matrix with blocks $\Lambda_1$ corresponding to the gradient $X'(t)$ and $\Lambda_2$ corresponding to the Hessian and the value $(X''(t),X(t))$.
\smallskip

The first consequence of isotropy is that
\begin{equation*} 
  \Lambda_1 = \lambda_2\I_{d},
\end{equation*}
being $\lambda_2$ the (now common) variance of the derivatives $X'_i(t):1\le i\le d$ or, equivalently, the second spectral moment of $X(t)$: $\lambda_2=\int_{\R^d}\nu_i^2f(\nu)\dd\nu$, $f$ being the spectral density of $\X$.
The second consequence of isotropy is that we can split $\Lambda_2$ into two blocks,
\begin{equation*} 
  \Lambda_2 = \left[\begin{array}{r|r}\Lambda_{2,O}&0\\ \hline 0&\Lambda_{2,\Delta}\end{array}\right],
\end{equation*}
since the off-diagonal second derivatives $X''_{O}(t):=(X''_{ij}(t):1\le i<j\leq d)$ are independent from the diagonal derivatives $X''_{\Delta}(t):=(X''_{ii}(t):1\le i\le d)$ and from $X(t)$.
Besides, the variance-covariance matrix of $X''_{O}(t)$ equals
\begin{equation*}
  \Lambda_{2,O} = \lambda_4\I_{d(d-1)/2},
\end{equation*}
with $\lambda_4$ the (now common) variance of the derivatives $X''_{ij}(t):1\le i,j\le d$ or, equivalently, the fourth spectral moment of $X(t)$: $\lambda_4=\int_{\R^d}\nu_i^4f(\nu)\dd\nu$.
Furthermore, the variance-covariance matrix of $(X''_{\Delta}(t),X(t))$ equals
\begin{equation*}
\Lambda_{2,\Delta}=\left[\begin{array}{c|c}V&0\\ \hline
  \ell 1^\top_d & \gamma\end{array}\right],
\end{equation*}
where $V=a_0\I_d-a_1\J_d$ with $a_0=\sqrt{\frac{2\lambda_4}{3}}$, $a_1=\sqrt{\frac{\lambda_4}{3}}\big(\frac{\sqrt{2}+\sqrt{d+2}}{d}\big)$, $\J_d$ the $d\times d$ matrix with all entries equal to $1$, $\ell=-\frac{\lambda_2}{a_0+da_1}$ and $1_d$ the all ones $d\times 1$ matrix.
In conclusion, the variance-covariance matrix of $\rX$ equals
\begin{equation*}
\Lambda = \left[\begin{array}{c|c|c|c}
  \lambda_2\I_{d} & 0 & 0 & 0 \\ \hline
  0 & \lambda_4\I_{d(d-1)/2} & 0 &0 \\ \hline
  0 & 0 & a_0\I_d - a_1 \J_d & 0\\ \hline
  0 & 0 & \ell 1^\top_d & \gamma\end{array}\right],
\end{equation*}

Finally, let us point out that the trace of the matrix $Y''(t)$ coincides with that of $X''(t)$ up to a constant:
\begin{equation} \label{e:trazas}
 {\rm tr}(Y''(t)) = (a_0+da_1)^{-1} {\rm tr}(X''(t)),
\end{equation}
and that the value $X(t)$ can be written as
\begin{equation} \label{e:XDY}
X_D(t) = \frac{\ell}{a_0+da_1}{\rm tr}(X''(t)) + \gamma Y_D(t).
\end{equation}
Indeed, from the above computations we have ${\rm tr}(Y'') = 1^\top_d Y''_{\Delta} = 1^\top_d V^{-1} X''_{\Delta} = (a_0+da_1)^{-1} 1^\top_d X''_{\Delta} = (a_0+da_1)^{-1} {\rm tr}(X'')$.
Here, we inverted $\Lambda$ block-wise:
\begin{equation*}
 \Lambda_{2,\Delta}^{-1}=\left[\begin{array}{c|c}V^{-1}&0\\\hline \widetilde{\ell} 1^\top_d & \gamma^{-1}\end{array}\right],
\end{equation*}
with $V^{-1}= \frac{1}{a_0}\I_d-\frac{a_1}{a_0(a_0+da_1)}\J_d$ and $\widetilde{\ell}=-\frac{\ell}{\gamma(a_0+da_1)}$.
\eqref{e:XDY} follows immediately from \eqref{e:trazas} and the expression for $\Lambda_{2,\Delta}$ above.
\section{Proofs of some lemmas} \label{a:somel}
\begin{proof}[Proof of Lemma \ref{l:pqcero}]
In the first place, note that condition (A4) implies that ${\mathbf f}_q$, as defined in \eqref{e:Fq}, charges a strictly positive measure on a small enough neighborhood of $0$ and, thus,  $P_q(0)=0$.
Denoting $b_{\balpha}({\bf m}) = \sum_{k=0}^d \alpha_kb_{k}({\bf m})$, we have
\begin{equation*}
   P_q(0) = \sum_{k=0}^d\alpha_k \Big(\sum_{{\bf m}\in[D]^q} b_{k}({\bf m}) \prod^q_{i=1} \psi_{m_{i}} (0)\Big).
\end{equation*}
From \eqref{e:nucleos-b} we have $\psi_m(\nu)=(\Lambda^{-1}\omega(\nu))_m$, since $(\omega(0))_m=0$ for $m<D$, we deduce that $\psi_m(0)=0$ for $m<D$.
Similarly, from \eqref{e:XDY} we see that $\psi_D(0)=\gamma^{-1}>0$. It follows that $\prod^q_{i=1} \psi_{m_{i}} (0)$ does not vanish only if $m_i=D$ for all $i$.
Hence,
\begin{equation*}
   P_q(0) = \psi_{D} (0)^q\sum_{k=0}^d \alpha_kb_{k}(D,\cdots,D)=0.
\end{equation*}
This equality gives the result.
\end{proof}

\begin{proof}[Proof of Lemma \ref{l:traza}]
For simplicity of notation, we omit $t$.
Using the joint distribution of $X''$ and $X$ described in \ref{a:pqcero} and conditioning on $X''$, we get
\begin{align} \label{e:cq}
 &0= \E\left[\sum_{k=0}^d(-1)^k\alpha_k \det(X'')\indicator_{A_k}(X'') \E\Big[\indicator_{[u,\infty)}(X)H_q(Y_D)\mid X''\Big]\right] \notag\\
 &=\E\left[\sum_{k=0}^d(-1)^k\alpha_k \det(X'')\indicator_{A_k}(X'') \E\Big[\indicator_{[u,\infty)}\big(\ell {\rm tr}(X'')+\gamma Y_D\big)H_q(Y_D)\mid X''\Big]\right] \notag\\
 &= \E\left[\sum_{k=0}^d(-1)^k\alpha_k \det(X'')\indicator_{A_k}(X'') \phi^{(q-1)}\Big(\frac{u-\ell {\rm tr}(X'')}{\gamma}\Big)\right] \notag\\
 &= \E\left[\sum_{k=0}^d(-1)^k\alpha_k \det(X'')\indicator_{A_k}(X'') (H_{q-1}\cdot\phi)\Big(\frac{u-\ell {\rm tr}(X'')}{\gamma}\Big)\right]
\end{align}
Here, we used \eqref{e:XDY} in the second line and Rodrigues' formula in the remaining ones.
Now, we show that we can restrict $v:=\gamma^{-1}(u-\ell{\rm tr}(X''))$ in \eqref{e:cq} to an arbitrary compact set.
Indeed, take $M>0$ and consider
\begin{multline*}
 \left|\E\left[\sum_{k=0}^d(-1)^k\alpha_k \det(X'')\indicator_{A_k}(X'') (H_{q-1}\cdot\phi)(v)\indicator_{|v|>M}\right]\right|\\
 \leq a\E\left[| \det(X'')| |H_{q-1}(v)|\phi(v)\indicator_{|v|>M}\right]\\
 = aC_{imk}\E\left[| \det(X'')|
\indicator_{|v|>M}\right],
\end{multline*}
where $a=\max\{|\alpha_k|:k=0,\cdots,d\}$ and $C_{imk}$ is the constant implied in Imkeller's inequality, see \cite[Prop. 3]{imkeller}. The r.h.s.
in the last display can be done arbitrariliy small as long as $M$ grows.
Hence, for all $q$
\begin{equation*}
 \E\left[\sum_{k=0}^d(-1)^k\alpha_k \det(X'')\indicator_{A_k}(X'') H_q(v)\phi(v) \indicator_{|v|\leq M}\right]=0.
\end{equation*}

Now, since this holds true for all Hermite polynomials in $v$, it holds true for any continuous function $g(v)$, that is, we can replace $H_{q-1}(v)\phi(v)\indicator_{|v|\leq M}$ in the last expression by $g(v)\indicator_{|v|\leq M}$.
Finally, taking $M\to\infty$ and redifining $g$, we get
\begin{equation*}
 \E\left[\sum_{k=0}^d(-1)^k\alpha_k \det(X'')\indicator_{A_k}(X'') g({\rm tr}(X''))\right]=0,
\end{equation*}
for $g\in C_0$.
By the very definition of conditional expectation, we have
\begin{equation*}
 \E\left[\sum_{k=0}^d(-1)^k\alpha_k \det(X'')\indicator_{A_k}(X'')\mid {\rm tr}(X'')\right]=0.
\end{equation*}
Thus, the result follows.
\end{proof}

\begin{proof}[Proof of Lemma \ref{l:LMethod}]
When $k=d$,
\[
G_d(z) =(-1)^d \E \left( \prod_{i=1}^d  (\nu_i -z)\indicator_{\nu_1<\nu_2<\cdots<\nu_d<z}\right).
\]
A direct application of the dominated convergence theorem yields that $ G_d(z)$ $\simeq z^d$, as $z \to \infty$.
Consider now the case $k<d$. We can write $G_k(z)$ as an integral
\begin{equation*}
 G_k(z) = \int_{\nu_1<\cdots<\nu_k<z<\nu_{k+1}<\cdots<\nu_d} (-1)^kI(\boldsymbol\nu)\dd\boldsymbol\nu,
\end{equation*}
with $I(\boldsymbol\nu)= K_d \prod^d_{i=1}(\nu_i -z)\prod_{i<j}(\nu_j-\nu_i)\prod^d_{i=1}\phi(\nu_i)$, where $K_d$ is the normalising constant associated to the density of the $\nu_i$'s.
We use the Laplace Method. The maximum of the exponential part of the integrand $I$ on the integration domain defined above is attained at the point $P=(0,\cdots,0;z,\cdots,z)$ $= z\sum^d_{i=k+1}e^d_i$.
The Laplace Method is based on a change of variable locally in a neighborhood of $P$.
We are allowed to replace any term in the integral by its equivalent at $P$.
We use the following change of variables which is in spherical coordinates for the last variables.
\[\begin{cases}
\nu_j \mapsto  u_j, & j=1,\cdots, k,\\
 \nu_j \mapsto z + \rho u_j, &j=k+1,\cdots, d, \quad  0<u_{k+1}< \cdots < u_d,
\end{cases}\]
where the vector $(u_{k+1},\cdots, u_d)$ belongs to a subset ${\cal R}$ of $\mathbb{S}^{d-k-1}$.
It is clear that, as $z \to +\infty $:
\[
 \prod_{i\leq k<j}(\nu_j-\nu_i) \simeq \ct z^{k(d-k)},
\]
and
\[
 \int_{\nu_1<\cdots<\nu_k<z} (-1)^k\prod^k_{i=1}(\nu_i-z) \prod_{i<j\leq k}(\nu_j-\nu_i)\prod^k_{i=1}\phi(\nu_i) \dd \nu_1  \ldots \dd \nu_k \simeq \ct z^k.
\]
So, it is sufficient to study the limit of the integral in the last variables
\[
\mathcal I:=\int_{z<\nu_{k+1}<\cdots <\nu_d } \prod_{i=k+1}^d(\nu_i-z)\prod_{k<i<j}(\nu_j-\nu_i)\prod_{i=k+1}^d\phi(\nu_i)\dd\nu_{k+1}\cdots \dd\nu_d.
\]
We set $\eta:=(d-k)(d-k+3)/2-1$, then
\begin{multline*}
  \mathcal I= \ct \int_{{\cal R}}\int^\infty_0 \rho^{\eta} \prod_{i=k+1}^d u_i\prod_{k<i<j}(u_j-u_i) \exp\left\{-\frac12\sum^d_{i=k+1}(z+\rho u_i)^2\right\}  \\ \dd\rho \dd\sigma(u_{k+1},\cdots,u_d)\\
  \simeq \ct e^{-\frac{d-k}{2}z^2} \int_{{\cal R}} \prod_{i=k+1}^d u_i\prod_{k<i<j}(u_j-u_i) \int^\infty_0 \rho^{\eta} \exp\left\{-\rho z \sum^d_{i=k+1}u_i\right\} \\ \dd\rho \dd\sigma(u_{k+1},\cdots,u_d)\\
  = \ct \frac{e^{-\frac{d-k}{2}z^2}}{z^{\eta+1}} \int_{{\cal R}} \frac{\prod_{i=k+1}^d u_i\prod_{k<i<j}(u_j-u_i)}{(\sum^d_{i=k+1}u_i)^{\eta+1}} \dd\sigma(u_{k+1},\cdots,u_d)
  \int^\infty_0 \rho^{\eta} e^{-\rho} \dd\rho \\
  =\ct \frac{e^{-\frac{d-k}{2}z^2}}{z^{\eta+1}},
\end{multline*}
where $\sigma$ denotes the surfacic measure on the sphere $\mathbb{S}^{d-k-1}$.
The result follows.
\end{proof}

\section{Table of Notation} \label{app:notation}
 \begin{table}[h!]
    \centering
    \renewcommand{\arraystretch}{1.0}
    \begin{tabular}{ll}
        \hline
        \textbf{Notation} & \textbf{Description} \\
        \hline
        $\balpha$ & The vector $(\alpha_0, \cdots, \alpha_d) \in \R^{d+1}$ \\
        $\cA_k$ & The set of symmetric $d\times d$ matrices with index $k$ \\
        $\hB$ & Complex Hermitian Brownian measure on $\R^d$ \\
        $\cC^{\balpha}(u,\cT)$ & Normalized linear combination of critical points \\
        $\crit^k (u,\cT)$ & Number of critical points in $\cT$ with index $k$ above level $u$ \\
        $\ct$ & Unimportant positive constant, its value may change from line to line \\
        $d$, $D$ & Dimensions related by $D=\frac{(d+1)(d+2)}{2}$ \\
        $e^{\mathcal D}_k$ & $k$-th canonical vectors in $\R^{\mathcal D}$\\
        $\cal F$ & Fourier transform \\
        $\Gamma(\cdot)$ & Covariance function of the last coordinate of $\rY(\cdot)$: $Y(\cdot)$ \\
        $\mathcal{H}$ & Set of complex-valued Hermitian square integrable functions \\
        $H_m$ & Hermite polynomial of degree $m$ \\
        $H_{\otimes {\boldsymbol n}}$ & Tensorial Hermite polynomial of degrees given by the coordinates of ${\boldsymbol n}$ \\
        $i(M)$ & Index of a symmetric matrix $M$ (number of its negative eigenvalues) \\
        $I^{\hB}_q$ & $q$-fold multiple Wiener-It\^o integral \\
	$\J_d$, $1_d$ & The all ones $d \times d$ matrix, $d \times 1$ vector\\
        $[q]$, ${\mathcal J}_q$ & The sets $\{1,\cdots,q\}$ and $\{{\mathbf n}\in\N^D:|{\mathbf n}|=n_1+\cdots+n_D=q\}$ respectively \\
        $\mathcal{K}_q$ & $q$-th Wiener chaos \\
        $\Lambda$ & Block matrix such that $\Lambda\Lambda^\top = \Var(\rX(t))$ \\
        N. D. & Non-degenerate \\
        $\phi_k$ & Standard normal density in $\R^k$ \\
        $\Psi(t)$ & Maximum bound on derivatives of $r(t)$ \\
        $r(\cdot)$ & Covariance function of $X(\cdot)$ \\
        $\mathbb{S}^{k-1}$ & Unit sphere in $\R^k$ \\
        $Sym(\psi)$ & Symmetrization of a kernel $\psi$ \\
        $\cT$ & The set $[-T,T]^d$ \\
        $\Var$ & Variance-covariance matrix  \\
        $W(\ve)$ & Volume of a ball of radius $\ve$ in $\R^d$ \\
        $X(\cdot)$ & Real-valued centered stationary Gaussian random field \\
        $\rX(t), \rY(t)$ & The vector $(X'(t), X''(t), X(t))$ and its standardized version \\
      $\|\cdot\|$ & Norms for matrices and order 3 and 4 tensors \\
        $\simeq$ & Equivalence of functions \\
        $\indicator_A$ & Characteristic, a.k.a. indicator, function of the set $A$\\
        $\bar{\otimes}_r$ & $r$-th contraction operator \\
        \hline
    \end{tabular}
\end{table}
\thispagestyle{empty}
\end{appendix}

\end{document}